\RequirePackage{ifpdf}
\ifpdf % We are running pdfTeX in pdf mode
\documentclass[pdftex]{sigma}
\else
\documentclass{sigma}
\fi
\usepackage{amsthm,cite}

\usepackage[english]{babel}
\usepackage{mathtools}

\newtheorem{thm}{Theorem}[section]
\newtheorem{lem}[thm]{Lemma}
\newtheorem{cor}[thm]{Corollary}
\newtheorem{prop}[thm]{Proposition}
\newtheorem{rem}[thm]{Remark}
\newtheorem{defn}[thm]{Definition}

\newcommand{\be}{\begin{equation}}
\newcommand{\ee}{\end{equation}}
\newcommand{\bea}{\begin{eqnarray}}
\newcommand{\eea}{\end{eqnarray}}
\newcommand{\Sinhq}{\text{Sinh}_q}

\newcommand{\E}{\varepsilon_q}
\newcommand{\qbinom}[2] {{#1 \brack #2}_q}

\newcommand{\re}{\text{Re}\,}
\newcommand{\Sin}{\text{Sin}}
\newcommand{\Cos}{\text{Cos}}
\begin{document}

%\renewcommand{\PaperNumber}{***}

%\FirstPageHeading

\ShortArticleName{On $q$-analogs of zeta functions}

\ArticleName{On $q$-analogs  of  zeta functions associated  with  a pair of $q$-analogs of Bernoulli numbers and polynomials}

% Names of the authors for the title of the paper
\Author{Ahmad  El-Guindy and  Zeinab Mansour}

\AuthorNameForHeading{A.M. El-Guindy and  Z.S.I. Mansour }

\Address{Mathematics Department, Faculty of Science, Cairo University, Giza, Egypt}

\Email{\href{mailto:elguindy@sci.cu.edu.eg}{aelguindy@sci.cu.edu.eg},
\href{mailto:zeinab@sci.cu.edu.eg}{zeinab@sci.cu.edu.eg}} % E-mail address of Second Author

% In the case of the same organization, please use the following standard
%\Author{First Names LASTNAME and Second COAUTHOR}
%\AuthoqNameForHeading{F.N. Lastname and S. Coauthor}
%\Address{Address of Author(s), Country}
%\Email{\href{mailto:email@address}{email1@address}, \href{mailto:email@address}{email2@address}}
%\URLaddress{\url{URL1}, \url{URL2})

%\ArticleDates{Received ???, in final form ????; Published online ????}

%\date{}%
%\dedicatory{}%
%\commby{}%
% ----------------------------------------------------------------
\Abstract{
In this paper, we use two different approaches  to  introduce  $q$-analogs of Riemann's zeta function and prove that their  values at  even integers are related to the $q$-Bernoulli and $q$ Euler's numbers introduced by Ismail and Mansour [Analysis and Applications, {\bf{17}}, 6, 2019, 853--895]. 
}
\Classification{11B68, 33E99} %11M35 ,   33E99,, 37C30}

% ----------------------------------------------------------------
\section{Introduction }

The Riemann zeta function $ \zeta(s)$ is the analytic function of the complex variable $s$, defined 
 for $\re(s)>1$  by 
$\zeta(s)=\sum_{k=1}^{\infty}\frac{1}{k^s}.$
 It  extends  on the complex plane to a meromorphic function with only  one 
 simple pole at $z=1$  with residue 1. 
  Moreover, 
for $n\in\mathbb{N}$
\be \label{Id0}
\zeta(2n)=2^{2n-1} (\pi)^{2n}(-1)^{n-1}\frac{\beta_{2n}}{2n!}
= (-1)^{n-1}\frac{2^{2n-2}}{1-4^n}(\pi)^{2n}\frac{{E}^*_{2n-1}}{(2n-1)!},
\ee
where $(\beta_n)_n$ are the classical  Bernoulli numbers and $(E_n^*)_n$ are the first-kind Euler numbers, see \cite{Apostol,Titchmarsh-Brown}. The alternating zeta function  or the Dirichlet eta function  is defined by
$\eta(s)=\sum_{k=1}^{\infty}\frac{(-1)^{k-1}}{k^s},\quad \re s>0$, see ~\cite{Sondow}.
It is known that $\eta(s)=(1-2^{1-s})  \zeta(s)$.  Both of  $\zeta(s)$ and $\eta(s)$ have  integral  and contour integral representations in terms of the generating functions of Bernoulli and Euler polynomials, see~\cite{Apostol}.  Using those representations,  many   identities can be proved, such as the identities  in \eqref{Id0} and the functional equation of the zeta function. In this paper, we introduce $q$-analogs of the zeta and eta functions.  Generally speaking, if $f_q(z)$ is a family of functions defined for $0<q<1$ (and $z$ in some domain), and if there is a function $f(z)$ for which $\lim_{q\to 1^-}f_q(z)=f(z)$, then we say that $f_q$ is  a $q$-analog of $f$. Usually there are many $q$-analogs of a given function, each one capturing different aspects of the function's properties. Indeed there has been many  $q$-analogs of the zeta function; see for example~
\cite{ Masato, Pup, Kenichi, Kvit95-a,Kvit95-b}, and also \cite{Srivastava-Choi} and the references therein.  In this paper, we propose and study another $q$-analog of the zeta function which has special values given in terms of $q$-Bernoulli  and $q$-Euler numbers introduced in \cite{Ism-Man-2019} such that identities similar to \eqref{Id0} hold. Furthermore our $q$-analog has associations with the so-called spectral zeta functions. Indeed Riemann's zeta function can be viewed (up to simple constants) as the spectral zeta function associated with the zeros of certain Bessel functions with half integer parameters (see \cite{Kishore} for instance). It is thus natural to explore  a $q$-generalization involving zeros of $q$-Bessel functions; this was indeed studied by Kvitsinsky (see \cite{Kvit95-a, Kvit95-b} for instance). Our approach is based in part on Kvitsinsky's studies, but we modify the definition of the resulting zeta function in such a way as to produce elegant formulas relating our proposed $q$-analog to $q$-Bernoulli numbers. The details are presented in section 4.
Unless otherwise  stated, $q$ shall denote positive number less than one.  We follow \cite{GR} for the basic definitions for $q$-functions, like the $q$-shifted factorial, the $q$-Gamma function, and  Jackson's $q$-analogs of the Bessel functions. Recall the two $q$-analogs of the exponential functions are defined by 
\[E_q(z)=(-(1-q)z;q)_{\infty} \;(z\in\mathbb{C})\;\mbox{and}\;e_q(z)=\frac{1}{(z(1-q);q)_{\infty}}\; (|z(1-q)|<1).\]
The $q$-trigonometric   functions $\text{Sin}_q z$ and  $\text{Cos}_q z$  are are defined in terms of $E_q(\pm iz)$ as in the classical case, see \eqref{S-C-def} below,  and we shall denote their positive zeros by $(\xi_k)_{k=1}^{\infty}$ and $(\eta_k)_{k=1}^{\infty}$, respectively.  
We introduce two  $q$-analogs of the zeta function which we define by 
\be\label{zetadefinitions} \begin{split}
\zeta_q(s)=\sum_{k=1}^{\infty}\frac{\text{Cos}_q\xi_k}{\text{Sin}'_q \xi_k} \frac{1}{\xi_k^s},\quad 
\zeta_q^*(s)=-\sum_{k=1}^{\infty}\frac{\text{Sin}_q\eta_k}{\text{Cos}'_q \eta_k} \frac{1}{\eta_k^s},
\end{split}
\ee
where $\re s>0$. Similarly, we introduce two $q$-analogs of the Dirichlet eta function  by 

\be \label{etadefinitions1} \begin{split}
\eta_q(s)=-\sum_{k=1}^{\infty}\frac{1}{\text{Sin}'_q \xi_k} \frac{1}{\xi_k^s},\quad 
\eta_q^*(s)=\sum_{k=1}^{\infty}\frac{1}{\text{Cos}'_q \eta_k} \frac{1}{\eta_k^s},
\end{split}
\ee
where $s\in \mathbb{C}$. 

%We shall show that the domain of  these two $q$-zeta functions is  $\re s>1$ and the domain of the two $q$-eta functions is $\mathbb{C}$ and that
We shall show that
\[\lim_{q\to 1^-}\zeta_q(s)=\frac{1}{\pi^s}\zeta(s),\quad  \lim_{q\to 1^-}\zeta_q^{*}(s)=\frac{2^s-1}{\pi^s}\zeta(s),\]
\[\lim_{q\to 1^-}\eta_q(s)=\frac{1}{\pi^s}\eta(s),\quad  \lim_{q\to 1^-}\eta_q^{*}(s)=\frac{2^s-1}{\pi^s}\eta(s).\]
{Thus as $q\to1^-$, the above expressions tend to simple multiples of $\zeta(s)$ and $\eta(s)$, respectively, which makes them valid $q$-analogs. However, quotients such as $\frac{\Cos_q(\xi_k)}{\Sin'_q(\xi_k)}$ have a more complicated behavior than their classical counterpart (which is simply equal to 1!). We devote a good part of section 3 to study the growth behavior of functions appearing in those quotients, culminating in the deduction of the convergence regions for the series stated above, cf. Lemma \ref{lem:bounded} and Corollary \ref{Cor:B}.}

We shall also show that the values of the functions we introduced  at even integers are related to the $q$-Bernoulli and Euler numbers introduced in~\cite{Ism-Man-2019}. 
Since the zeros of  $\text{Sin}_q z$ and $\text{Cos}_qz$ are real and separate each other,  see \cite{Ismail-82},  one can verify that
$\text{Cos}_q\xi_k$ and $\text{Sin}_q'\xi_k$ have the same sign while $\Sin_q\eta_k$ and $\Cos_q'\eta_k $ have opposite signs, which is the reason we have a negative sign in the definitions of  $\zeta_q^* (s)$ and $\eta_q(s)$.

This paper is organized as follows.  Section 2  contains some preliminaries on $q$-analogs of the exponential, trigonometric, and Bessel functions. In section 3, some asymptotic properties of those functions are studied with an eye towards convergence questions and other applications in later sections. In Section 4, we recall the definition of spectral zeta functions and study those related to the zeros of the second $q$-Bessel functions.  We obtain a $q$-analog of the Rayleigh functions and derive two new $q$-analogs of the zeta function.  In Section 5 we introduce a $q$-analog of Hurwitz's zeta function and Dirichlet's eta function by using contour integration for a function of the form $z^{s-2}f(z)$ where $f$ is either the generating function of the $q$-Bernoulli polynomials or $q$-Euler polynomials. In Section  6,  we define another $q$-analog of Dirichlet's eta function and prove some identities relating these $q$-analogs of the Riemann zeta and Dirichlet eta functions. Lastly, in section 7 we show that the contour integration approach can also be used to work with the $q$-analog of the zeta function.

\section{Preliminaries}
We follow \cite{Ism-Man-2019} for the definitions
 of  the exponential functions $E_q(z)$, $e_q(z)$ and \cite{GR} for the definition and properties of the second Jackson's $q$-Bessel functions $J_\nu^{(2)}(x;q)$. 
Jackson~\cite{Jack} defined  $q$-analogs of the trigonometric functions by  
\be\label{S-C-def}
\begin{split}
\text{Sin}_q z&=\dfrac{E_q(iz)-E_q(-iz)}{2i}=\sum_{k=0}^{\infty} (-1)^k\frac{q^{k(2k+1)}}{\Gamma_q(2k+2)} z^{2k+1},\\
\text{Cos}_qz&=\dfrac{E_q(iz)+E_q(-iz)}{2}=\sum_{k=0}^{\infty} (-1)^k\frac{q^{k(2k-1)}}{\Gamma_q(2k+1)} z^{2k}.
\end{split}
\ee
These $q$-trigonometric functions are connected to certain $q$-Bessel functions of half-integer orders by the identities 
\be\begin{split}\label{ID00}
	J_{1/2}^{(2)}(x;q^2)&= \frac{1}{\Gamma_{q^2}(1/2)} (\frac{2}{x(1-q^2)})^{1/2} \text{Sin}_q(\frac{x}{2(1-q)}),\\
  J_{3/2}^{(2)}(x;q^2)&=\frac{q^{-1}}{\Gamma_{q^2}(1/2)} (\frac{2}{x(1-q^2)})^{1/2}\left[\frac{2(1-q)}{x}\text{Sin}_q(\frac{x}{2(1-q)})-\text{Cos}_q( \frac{x}{2(1-q)})\right].
  \end{split}\ee
 Ismail~\cite{Ismail}  proved  that the  zeros of $ J_{\nu}^{(2)}(\cdot;q)$ are real  and simple, and the zeros of
  $x^{-\nu}J_{\nu}^{(2)}(x;q)$ and $x^{-\nu-1}J_{\nu+1}^{(2)}(x;q)$ are interlacing and simple.  Therefore,  
 the  zeros of $\text{Sin}_q x$ and $\text{Cos}_qx $ are infinite, real, simple,  and interlacing.  Their only cluster point is infinity.  Kvitsinsky ~\cite{Kvit95-b} proved that the positive  zeros $j_{k,\nu}(q^2)$ of $J_{\nu}^{(2)}(\cdot;q^2)$ have the asymptotic
 \be
 j_{k,\nu}(q^2)\sim 2 q^{-\nu-1-2k}\quad \text{as}\quad  k\to \infty.
 \ee
 Therefore, the positive zeros $\xi_k$ of $\Sin_qz$ have the asymptotics   
 \be\label{Ass}\xi_k \sim A q^{-2k},\;A:=\frac{q^{-3/2}}{1-q} \quad\text{as}\quad  k\to \infty,\ee
 and the positive zeros $\eta_k$ of $\Cos_q z$ have the asymptotics
  \be\label{Ass-2}\eta_k \sim B q^{-2k},\;B:=\frac{q^{-1/2}}{1-q} \quad\text{as}\quad  k\to \infty,\ee
  Similar studies for the asymptotics of the zeros of the third Jackson $q$-Bessel functions can be found in 
  \cite{Annaby-Mansour-009}. 
In \cite{NP,Ism-Man-2019}, the authors defined  a pair of $q$-analogs of  the Bernoulli polynomials through the generating functions 
\be\label{GF}\begin{split}
\frac{te_q(tx)}{e_q(t/2)E_q(t/2)-1}&:=\sum_{n=0}^{\infty} b_n(x;q)\frac{t^n}{[n]!},\\
\dfrac{t\,E_q(xt)}{E_q(t/2)e_q(t/2)-1}&:=\sum_{n=0}^{\infty}B_n(x;q)\frac{t^n}{[n]!}.
\end{split}
\ee
 The $q$-Bernoulli numbers are defined by
\be\label{GF0}\beta_n(q):=B_n(0;q)=b_n(0;q).\ee
 
They also  introduce the following pair of $q$-analogs of Euler's polynomials
\[\begin{split}
\frac{2e_q(xt)}{e_q(t/2)E_q(t/2)+1}&=\sum_{n=0}^{\infty}e_n(x;q)\frac{t^n}{[n]!},\\
 \frac{2E_q(xt)}{e_q(t/2)E_q(t/2)+1}&=\sum_{n=0}^{\infty}E_n(x;q)\frac{t^n}{[n]!}.
\end{split}\]
 The $q$-Euler numbers $ \widetilde{E}_n$ of the first kind  are defined by
$\widetilde{E}_n=E_n(0;q)=e_n(0;q).$
A pair of  $q$-analogs of the  Euler's numbers are  defined by
\[E_n:=2^n E_n(1/2;q),\quad e_n:=2^n e_n(1/2;q).\] 
A $q$-analog of Genocchi numbers is defined through the generating function
\[\frac{2t}{e_q(t/2)E_q(t/2)+1}=\sum_{n=0}^{\infty}G_n(q) \frac{t^n}{[n]!}.\]
Simple manipulation shows that
$G_n(q)=[n]\widetilde{E}_{n-1}(q),$
\[\beta_0=1,\; \beta_1=-1/2,\;  \beta_{2k+1}=0,\;k=1,2,\ldots,\] and 
\[\widetilde{E}_0=1,\; \widetilde{E}_1=-1/2,\; \widetilde{E}_{2k}=0,\; k=1,2,\ldots.\]

It is worth noting that in \cite{Cies}, Cie\'{s}li\'{n}ski defines another $q$-analog of the exponential function by 
\be\label{ne} \E(z):=e_q(z/2)E_q(z/2),\quad z\in\mathbb{C},\ee
which already appears in the generating functions of Euler and Bernoulli polynomials and numbers.

\section{Asymptotic Properties }
 This section includes asymptotic properties that we shall utilize when studying the convergence of contour integrals appearing in Sections~\ref{Sec:Dir-Bernoulli} and   \ref{sec:Der-Euler}.  {Most of our results and estimates have well-known classical counter-parts. However, the proofs in the $q$-setting are usually more involved since the $q$-exponential function $E_q$ has infinitely many real zeros.}
\begin{lem}\label{lemma1}
For $x>0$, the function 
$
\frac{E_q(x)}{\Sinhq(x)}
$
is  bounded.
\end{lem}
\begin{proof}
Write $r=x(1-q)$, then we need to show that
\[
\dfrac{E_q(r)}{2\text{Sinh}_qr }=\frac{(-r,q)_\infty}{(-r;q)_\infty-(r;q)_\infty}
\]
is bounded as $r \to \infty$. Write $r=q^{-n+\delta}$ with $n$ is a positive integer and $0\leq \delta<1$. Then we have
\[
(-q^{-n+\delta};q)_\infty=q^{-\binom{n+1}{2}+n\delta}(-q^{1-\delta};q)_n (-q^\delta;q)_\infty.
\]
Also 
\[
(-r;q)_\infty-(r;q)_\infty=\sum_{j=0}^\infty \frac{q^{(2j+1)(j-n+\delta)}}{(q;q)_{2j+1}}.
\]
If we choose $j=\lfloor \frac{n}{2}\rfloor$,  we get that the above expression is greater than 
$\frac{q^{-\binom{n+1}{2}}}{(q;q)_{n+1}}q^{(n+1)\delta}$ if $n$ is even, and greater than $\frac{q^{-\binom{n+1}{2}}}{(q;q)_{n}}q^{n\delta}$ if $n$ is odd. Noting that $q^\delta<1$ and $(q;q)_n>(q;q)_{n+1}$ we get
\[
\frac{(-r,q)}{(-r;q)-(r;q)}\leq q^{-\delta}(q;q)_n(-q^{1-\delta};q)_n (-q^\delta;q)_\infty\leq q^{-\delta}(-q^{1-\delta};q)_\infty(-q^\delta;q)_\infty,
\]
which is bounded above by $q^{-1}(-1;q)_\infty^2$ for all $n, \delta$ and hence for all $r$.

\end{proof}

\begin{lem}\label{lem:1}
Let $\{\xi_k\}_{k=1}^{\infty}$ (resp. $\{\eta_k\}_{k=1}^{\infty}$) be the positive zeros of $\text{Sin}_qz$ (resp. $\Cos_q z$). 
Let $\alpha$ and $\beta$ be real positive numbers such that $\alpha+\beta<2$. Then there exists $n_0\in\mathbb{N}$ such that for all $n\geq n_0$
\bea \label{Eq:ineq}q^{-\alpha}\xi_n< q^{\beta}\xi_{n+1},\\
 q^{-\alpha}\eta_n< q^{\beta}\eta_{n+1}.\eea
\end{lem}
\begin{proof}
The proof follows from \eqref{Ass}  since  $\lim \frac{\xi_{n+1}}{\xi_n}=q^{-2}$. Hence if we choose 
$\epsilon=1-q^{2-(\alpha+\beta)}$,  then  there exists $n_0\in\mathbb{N}$ such that 
\be \label{epsion}(1-\epsilon)<\frac{\xi_{n+1}}{\xi_n}q^{2}<(1+\epsilon),\;\mbox{for all}\; n\geq n_0,\ee
and \eqref{Eq:ineq} follows. The proof for $\eta_k$ is similar and omitted for brevity.
\end{proof}

%Merged the lemma for $\eta$

\begin{lem}\label{lem:2}
Let $\alpha$, $\beta$, and $n_0$ be as in \text{Lemma }\ref{lem:1}.  Let $A$ be the constant in \eqref{Ass} and $\gamma$ is any positive constant less than $\alpha$ and $\beta$. Then for  $q^{-\alpha}\xi_n<R_n<q^{\beta}\xi_{n+1}$,
%\be\label{In0} \ee
\bea  
\frac{1}{|\text{Sinh}_q z|}=O\left(\left(\frac{Aq^{\beta}}{(1+q^{\gamma})}\right)^n\,R_n^{-n}\right)
, \;|z|=R_n, \;n\to\infty,\\
\frac{1}{|\text{Cosh}_q z|}=O\left(\left(\frac{Bq^{\beta}}{(1+q^{\gamma})}\right)^n\,R_n^{-n+1}\right)
, \;|z|=R_n, \;n\to\infty.
\eea

\end{lem}

\begin{proof}From \eqref{Ass}, if 

 $0<\gamma<\min\{\alpha,\beta\}$,  and $\epsilon=\frac{1-q^{\gamma}}{1+q^{\gamma}}$, then there exists $m_0\in\mathbb{N}$ such that
 \be \label{In}(1-\epsilon)A q^{-2n}<\xi_n<(1+\epsilon)Aq^{-2n},\;\mbox{for all}\; n\geq m_0.\ee
 Take $k_0=\max\{n_0,m_0\}$. Then 
\be\label{Id:1}
\begin{split}
|\text{Sinh}_qz|&=|z|\left|\prod_{k=1}^{\infty}1+\frac{z^2}{\xi_k^2}\right|\\
&=|z|\prod_{k=1}^{n}|1+\frac{z^2}{\xi_k^2}|\prod_{k=n+1}^{\infty}{|1+\frac{z^2}{\xi_k^2}|}.
\end{split}
\ee
Now if $|z|=R_n$, then from $q^{-\alpha}\xi_n<R_n<q^{\beta}\xi_{n+1}$ we get
\[\begin{split}\left|\prod_{k=1}^{n}1+\frac{z^2}{\xi_k^2}\right|&\geq \prod_{k=1}^{n}\frac{q^{-2\alpha}{\xi_n}^2}{\xi_k^2}-1
= \prod_{k=1}^{n}\frac{q^{-2\alpha}{ \xi_n}^2}{\xi_k^2}\prod_{k=1}^{n}1-\frac{q^{2\alpha }{\xi_k}^2}{\xi_k^n}\\
&\geq  \prod_{k=1}^{n}\frac{q^{-2\alpha}{ \xi_n}^2}{\xi_k^2} \prod_{k=1}^{k_0-1}1-\frac{q^{2\alpha}{\xi_k}^2}{\xi_n^2} \prod_{k=k_0}^{n}\left(1-\frac{q^{2\alpha}{ \xi_k}^2}{\xi_n^2}\right).
\end{split}
 \]
Since 
\be \label{Eq:ineq1} \begin{split}\prod_{k=k_0}^{n}\left(1-\frac{q^{2\alpha}{ \xi_k}^2}{\xi_k^n}\right)&\geq
\prod_{k=k_0}^{n}\left(1-q^{2\alpha}\frac{(1+\epsilon)^2}{(1-\epsilon)^2}q^{4n-4k}\right)\\
&=\prod_{k=k_0}^{n}\left(1-q^{2\alpha-2\gamma}q^{4n-4k}\right)
\geq (q^{2\alpha-2\gamma};q^4)_{\infty},
\end{split}\ee
\be \label{Eq:Ineq2}\prod_{k=1}^{k_0-1}1-\frac{q^{2\alpha}{\xi_k}^2}{\xi_n^2} \geq (1-q^{2\alpha})^{k_0-1},\ee
 
\be\label{Eq:Ineq3} \prod_{k=1}^{n}q^{-2\alpha}\frac{\xi_n^2}{\xi_k^2}\geq \prod_{k=k_0}^{n}q^{-2\alpha}\frac{\xi_n^2}{\xi_k^2}\geq q^{-2(\alpha-\gamma)(n-k_0+1)}q^{-2(n-k_0)(n-k_0+1)}.
\ee
and  
\be\label{Eq:Ineq4} \prod_{k=n+1}^{\infty}{|1+\frac{z^2}{\xi_k^2}|}\geq (q^{4+2\beta-2\gamma};q^4)_{\infty}.\ee
Combining \eqref{Id:1}, \eqref{Eq:ineq1} --\eqref{Eq:Ineq4}, we obtain
\[\begin{gathered}|\text{Sinh}_qz|\geq  R_n (1-q^{2\alpha})^{k_0-1}q^{-2(\alpha-\gamma)(n-k_0+1)}q^{-2(n-k_0)(n-k_0+1)}\\ \times (q^{2\beta-2\gamma+4};q^2)_{\infty}(q^{2\alpha-2\gamma};q^4)_{\infty},\end{gathered}\]
where $|z|=R_n$.
Using $q^{-\alpha}\xi_n<R_n<q^{\beta}\xi_{n+1}$ and \eqref{In}, we can prove that 

\be \frac{q^{\alpha}}{A(1-\epsilon)}R_n>q^{-2n}>\frac{q^{2-\beta}}{A(1+\epsilon)}R_n,\;n\geq k_0.  \ee

Consequently, for $|z|=R_n$,
\[\frac{1}{\left|\text{Sinh}_q z\right|}\leq C(\alpha,\beta,\gamma,A) \left(\frac{Aq^{\beta}}{1+q^{\gamma}}\right)^n  R_n^{k_0-n},\;n\geq k_0\]
where 
\[C(\alpha,\beta,\gamma,A)=(1-q^{2\alpha})^{-k_0+1} q^{2(\alpha-\gamma-k_0)(-k_0+1)} 
\left(\frac{2Aq^{\beta-2}}{1+q^{\gamma}}\right)^{\alpha-\gamma+1} \left(\frac{q^{\alpha-\gamma}(1+q^{\gamma})}{2A}\right)^{k_0}.\]
The proof for $\text{Cosh}_q$ is similar and omitted for brevity.
\end{proof}

%\begin{lem}\label{lem:2-2}
%Let $\alpha$, $\beta$, and $n_0$ be as in \text{Lemma \ref{lem:1-1}}.  Let $A$ be the constant in \eqref{Ass-2} and $\gamma$ is any positive constant less than $\alpha$ and $\beta$. Then for  
%\be q^{-\alpha}\eta_n<R_n<q^{\beta}\eta_{n+1},\ee

%\end{lem}
%\begin{proof}
%The proof is similar to the proof of \ref{lem:2} and is omitted.
%\end{proof}

From now on,  we use  the notations  
\[m(f,R):=\min\{|f(z)|,\; |z|=R\}\quad \text{and}\quad M(f,R):=\max\{|f(z)|,\;|z|=R\}. \]

\begin{cor}
Let $R_n$ be as in Lemma \textup{\ref{lem:2}}, then for any $s\in\mathbb{C}$,
\[\lim_{n\to\infty}M\left(\frac{z^s}{\text{Sinh}_qz},R_n\right)=0,\quad \lim_{n\to\infty}M\left(\frac{z^s}{\text{Cosh}_qz},R_n\right)=0 .\]
\end{cor}

\begin{lem}\label{lem:3}
Let $\mu$ and $\nu$ be complex numbers such that $\nu\neq 0$ . Let $c$ and $d$ be positive real numbers such that $c+d<1$. Let $R_n$ be a sequence of positive numbers satisfying 
$\frac{q^{-c-n}}{|\nu|}<R_n<\frac{q^{d-n-1}}{|\nu|}$, and 
\[M_n:=M\left(\frac{(\mu\,z;q)_{\infty}}{(\nu z;q)_{\infty}},R_n\right).\]
Then 
\[M_n=O\left(\left(q^{c+d-1}\frac{|\mu|}{|\nu|}\right)^n\right)\quad \mbox{as}\quad n\to\infty.\]
\end{lem}

\begin{proof}
We can prove that if $|z|=R_n$ then
\[|(\nu z;q)_{\infty}|\geq q^{-c(n+1)}q^{-n(n+1)/2}(q^c;q)_{\infty}(q^d;q)_{\infty},\]
and
\[|(\mu z;q)_{\infty}|\leq \left(\frac{|\mu|}{|\nu|}\right)^{n+1} q^{(d-1)(n+1)}q^{-n(n+1)/2}(-\frac{|\mu|}{|\nu|}q^d;q)_{\infty}(-\frac{|\nu|}{|\mu|}q^{1-d};q)_{\infty}.\]
Therefore, 
\[M_n=\left|\frac{(\mu\,z;q)_{\infty}}{(\nu z;q)_{\infty}}\right|\leq q^{c+d-1} \left(q^{c+d-1}\frac{|\mu|}{|\nu|}\right)^n
\dfrac{(-\frac{|\mu|}{|\nu|}q^d;q)_{\infty}(-\frac{|\nu|}{|\mu|}q^{1-d};q)_{\infty}}{(q^c;q)_{\infty}(q^d;q)_{\infty}} .\]

\end{proof}

\begin{lem} \label{lem_3}Let $\alpha$ and $\beta$ be positive numbers satisfying   $\alpha+\beta<2$, and $\mathcal{A}_n$ be  the annulus defined  for  $n\in\mathbb{N}$ 
\[\mathcal{A}_n:=\{z\in\mathbb{C}:q^{-\alpha}\xi_n<|z|<q^{\beta}\xi_{n+1}\}.\]
Then, there exist positive constants 
  $c$ and $d$,     $c+d<1 $, an  integer $ k_0$, a positive integer $N_0$  such that the annulus 

\[\mathcal{B}_{2n+2k_0}:=\left\{z\in\mathbb{C}:\frac{q^{-c-2n-2k_0}}{|\nu|}<|z|<\frac{q^{d-2n-2k_0-1}}{|\nu|}\right\}\]
has a non empty intersection with $\mathcal{A}_n$ for all $n\geq N_0$.

\end{lem}

\begin{proof}
From \eqref{Ass}, $\xi_n\sim Aq^{-2n}$  as $(n\to\infty)$. Consequently,  for $\epsilon<\frac{q^{\alpha+\beta-2}-1}{q^{\alpha+\beta-2}+1}$, there exists 
$N_0\in\mathbb{N}$ such that 
\be \label{pf:1}q^{-2n}A(1-\epsilon)<\xi_n<q^{-2n}A(1+\epsilon)\quad (n\geq N_0).\ee
Let $c>0$ and $k_0\in\mathbb{Z}$ be chosen such that 
\[\dfrac{\log(q^{-\alpha }A|\nu|(1+\epsilon))}{\log 1/q}<c+2k_0<\dfrac{\log(q^{\beta-2}A|\nu|(1-\epsilon))}{\log 1/q}.\]Therefore, 
\[q^{-\alpha}A(1+\epsilon)<\frac{q^{-c-2k_0}}{|\nu|}\leq q^{\beta-2}A(1-\epsilon).\]
Consequently from \eqref{pf:1}
\[q^{-\alpha}\xi_n<q^{-2n-\alpha}A(1+\epsilon)<\frac{q^{-c-2n-2k_0}}{|\nu|}\leq q^{-2n+\beta-2}A(1-\epsilon)<q^{\beta}\xi_{n+1}\]
for all $n\geq N_0$. Thus $\mathcal{A}_n\cap \mathcal{B}_{2n+2k_0}\neq \emptyset.$
\end{proof}
\begin{cor}\label{Cor}Let   the constants $\alpha$, $\beta$, $c$, and $d$, $k_0$ and $N_0$ be as in   \textup{Lemma} \ref{lem_3}.   
Then  for $z\in \mathcal{A}_n\cap \mathcal{B}_{2n+2k_0}$   if $R_n:=|z|$, then 
\[\lim_{n\to\infty}R_n M\left(z^{s-1}\frac{e_q( \nu z)E_q(z)}{\text{Sinh}_q z}\right)=0.\]

\end{cor}

\vskip  .5 cm

Clearly,  the intersection of the two annuli described in the last corollary does not contain any poles for $f(z)$.

\begin{lem} \label{lem_3-3}Let $\alpha$ and $\beta$ be positive numbers satisfying   $\alpha+\beta<2$, and for  $n\in\mathbb{N}$ let $\mathcal{A}_n$ be the annulus defined by   
\[\mathcal{A}_n:=\{z\in\mathbb{C}:q^{-\alpha}\eta_n<|z|<q^{\beta}\eta_{n+1}\}.\]
Then, there exist positive constants 
  $c$ and $d$,     $c+d<1 $, an  integer $ k_0$, a positive integer $N_0$  such that the annulus 

\[\mathcal{B}_{2n+2k_0}:=\left\{z\in\mathbb{C}:\frac{q^{-c-2n-2k_0}}{|\nu|}<|z|<\frac{q^{d-2n-2k_0-1}}{|\nu|}\right\}\]
has a non empty intersection with $\mathcal{A}_n$ for all $n\geq N_0$.

\end{lem}
\begin{cor}\label{Cor-1} Let   the constants $\alpha$, $\beta$, $c$, and $d$, $k_0$ and $N_0$ be as in   \textup{Lemma} \ref{lem_3-3}.  
Then  for $z\in \mathcal{A}_n\cap \mathcal{B}_{2n+2k_0}$   if $R_n:=|z|$, then 
\[\lim_{n\to\infty}R_n M\left(z^{s-1}\frac{e_q( \nu z)E_q(z)}{\text{Sinh}_q z}\right)=0.\]

\end{cor}

\begin{lem}\label{lemma8}
Let $\epsilon$ be a positive number less than $\frac{1-q^2}{1+q^2}$, and let $C=\frac{1+\epsilon}{1-\epsilon}$. Then 
\be\label{lem8:1} \frac{1}{\left|\text{Sin}'_q\xi_n\right|}= O\left(C^{2n} q^{2n(n-1)} \right),\ee
\be\label{lem8:2} \frac{1}{\left|\text{Cos}'_q\eta_n\right|}= O\left(\frac{C^{2n}}{\eta_n} q^{2n(n-1)} \right).\ee
as $n\to\infty$.

\end{lem}
\begin{proof}
We prove \eqref{lem8:1} and the proof of \eqref{lem8:2} is similar and is omitted.
Let $0<\epsilon<\frac{1-q^2}{1+q^2}$.   From \eqref{Ass}, there exists $m_0\in\mathbb{N}$ such that 
 \be \label{In:0}(1-\epsilon)A q^{-2n}<\xi_n<(1+\epsilon)Aq^{-2n},\;\mbox{for all}\; n\geq m_0.\ee
 Let $C=\frac{1+\epsilon}{1-\epsilon}$. Then $1<C<q^{-2}$. Since 
 
\[\dfrac{\text{Sin}_qz}{1-\frac{z^2}{\xi_n^2}}=z\prod_{k=1}^{n-1}\left(1-\frac{z^2}{\xi_k^2}\right)\prod_{k=n+1}^{\infty}\left(1-\frac{z^2}{\xi_k^2}\right).\]
Taking the limit as $z\to \xi_n$, we obtain 
\[\text{Sin}'_q(\xi_n)=-2\prod_{k=1}^{n-1}\left(1-\frac{\xi_n^2}{\xi_k^2}\right)\prod_{k=n+1}^{\infty}\left(1-\frac{\xi_n^2}{\xi_k^2}\right).\]
Set \[K(\epsilon,m_0):=A^{2m_0+2} (1+\epsilon)^{2m_0}\dfrac{q^{-2m_0(m_0-1)}}{\prod_{k=1}^{m_0-1}\xi_k^2}.\]Using \eqref{In:0}
we can prove that 

\[\begin{split}\left|\prod_{k=1}^{n-1}1-\frac{\xi_n^2}{\xi_k^2}\right|&=\prod_{k=1}^{n-1}\frac{\xi_n^2}{\xi_k^2}\prod_{k=1}^{n-1}(1-\frac{\xi_k^2}{\xi_n^2})
=\dfrac{\prod_{k=1}^{n-1}\xi_n^2}{\prod_{k=1}^{m_0-1}\xi_k^2\;\prod_{k=m_0}^{n-1}\xi_k^2} \prod_{k=1}^{m_0-1}(1-\frac{\xi_k^2}{\xi_n^2})\prod_{k=m_0}^{n-1}(1-\frac{\xi_k^2}{\xi_n^2})\\
&\geq K(\epsilon,m_0) (\frac{1-\epsilon}{1+\epsilon})^{2n}q^{-2n(n-1)}(1-\frac{1}{\xi_n^2})^{m_0-1}\prod_{k=m_0}^{n-1}(1-\frac{\xi_k^2}{\xi_n^2})\\
&\geq K(\epsilon,m_0) (\frac{1-\epsilon}{1+\epsilon})^{2n}q^{-2n(n-1)}(1-\frac{1}{\xi_n^2})^{m_0-1} \left(\left(\frac{1+\epsilon}{1-\epsilon}\right)^2q^4;q^4\right)_{\infty}.
\end{split}\]
Consequently,
\be\label{In:2}\begin{split}
\left|\prod_{k=1}^{n-1}1-\frac{\xi_n^2}{\xi_k^2}\right|&\geq K(\epsilon,m_0)   C^{-2n}q^{-2n(n-1)}
(1-\frac{1}{\xi_n^2})^{m_0-1} \left(C^2q^4;q^4\right)_{\infty}.
\end{split}
\ee
Also
\be\label{In:3} \prod_{k=n+1}^{\infty}1-\frac{\xi_n^2}{\xi_k^2 }\geq \left(C^2q^4;q^4\right)_{\infty}.  \ee
Combining \eqref{In:2} and \eqref{In:3}, we obtain the required result.
\end{proof}
\begin{cor}
For any $s\in\mathbb{C}$ and for any $a>0$ , the series 
\be\sum_{k=0}^{\infty}\xi_k^{s-1}\frac{E_q(\pm i \xi_k)}{E_q(\pm 2a i\xi_k)Sin'_q\xi_k}\ee
is absolutely convergent.

\end{cor}
\begin{proof}
The proof is a direct consequence of Lemmas \ref{lemma1} and \ref{lemma8} and is omitted.
\end{proof}
\begin{lem}\label{lem:bounded}  The sequences 
 $\left(q^{2n}\frac{Cos_q\xi_n}{\text{Sin}'_q\xi_n}\right)_{n=1}^{\infty}$  and
 $\left(q^{2n}\frac{Sin_q\eta_n}{\text{Cos}'_q\eta_n}\right)_{n=1}^{\infty}$   are bounded.

\end{lem}

\begin{proof}
We prove the lemma for   $\left(q^{2n}\frac{Cos_q\xi_n}{\text{Sin}'_q\xi_n}\right)_{n=1}^{\infty}$. The proof of  the boundedness of  $\left(q^{2n}\frac{Sin_q\eta_n}{\text{Cos}'_q\eta_n}\right)_{n=1}^{\infty}$ is similar and is omitted.
Since, $D_q \text{Sin}_q x=\text{Cos}_q qx $, then we can prove that 
\[\text{Cos}_q \xi_k=\frac{q}{(1-q)\xi_k} \text{Sin}_q q^{-1}\xi_k=\frac{q}{1-q}\prod_{j=1}^{\infty}\left(1-\frac{q^{-2}\xi_k^2}{\xi_j^2}\right).\]
Let $0<\epsilon<\frac{1-q^2}{1+q^2}$, and $C=\frac{1+\epsilon}{1-\epsilon}$. Then from \eqref{Ass},   there exists $k_0\in\mathbb{N}$ such that 
\[k\geq k_0\to (1-\epsilon)A q^{-2k}<\xi_k<(1+\epsilon)Aq^{-2k}.\]
Consequently, 
\[\text{Cos}_q \xi_n=\frac{q}{1-q}\prod_{k=1}^{n}\left(1-\frac{q^{-2}\xi_n^2}{\xi_k^2}\right)\prod_{k=n+1}^{\infty}\left(1-\frac{q^{-2}\xi_n^2}{\xi_k^2}\right).\]
But 
\[\begin{split}\prod_{k=1}^{n}\left(1-\frac{q^{-2}\xi_n^2}{\xi_k^2}\right)&=\frac{q^{-2n}\xi_n^{2n}}{\prod_{k=1}^{n}\xi_k^2} \prod_{k=1}^{k_0-1}\left(1-\frac{q^{2}\xi_k^2}{\xi_n^2}\right)\prod_{k=k_0}^{n}\left(1-\frac{q^{2}\xi_k^2}{\xi_n^2}\right)\\
&=O(q^{-2n}\frac{\xi_n^{2n}}{\prod_{k=k_0}^{n}\xi_k^2}).
\end{split}.\]
Hence, 
\[\text{Cos}_q \xi_n=O(C^{-2n} q^{-2n^2+2n})\;\text{ as}\; n\to \infty,\]  therefore from Lemma \ref{lemma8}, we get  
\[\left|\frac{\text{Cos}_q\xi_n}{\text{Sin}'_q\xi_n}\right|=O(q^{-2n})\]  for sufficiently large $n$.  Consequently, 
$\left(q^{2n}\frac{\text{Cos}_q\xi_n}{\text{Sin}'_q\xi_n}\right)_{n\in\mathbb{N}}$ is a bounded sequence.

\end{proof}
A direct consequence of Lemma \ref{lem:bounded} is the following corollary,  {which justifies the convergence bounds stated in the introduction in \eqref{zetadefinitions} and \eqref{etadefinitions1}}.  
\begin{cor}\label{Cor:B}
The series $\sum_{n=1}^{\infty}\frac{1}{\xi_n^s}\frac{\text{Cos}\xi_n}{\text{Sin}_q'\xi_n}
$ and $\sum_{n=1}^{\infty}\frac{1}{\eta_n^s}\frac{\text{Sin}\eta_n}{\text{Cos}_q'\eta_n}
$ are absolutely convergent for $
\re s>1$.

\end{cor}

 {Our final result of this section will be useful in section 5 where we derive a series representation of a $q$-analog of the Hurwitz zeta function (cf. Theorem \ref{thm:bn}).}
\begin{lem}\label{3:witha} 
For $a>0$ and $s\in\mathbb{C}$, the series 
\be \label{ser}\sum_{k=0}^{\infty}(-1)^k q^{k(k+1)/2}\left(\frac{q^{-k}}{a(1-q)}\right)^{s-1} \frac{E_q\left(-\frac{q^{-k}}{2a(1-q)}\right)}{\text{Sinh}_q(\frac{q^{-k}}{2a(1-q)})}\ee
is absolutely convergent. 

\end{lem}
\begin{proof}
Using the identity,  \[(cq^{-n};q)_n=(q/c;q)_n \left(-\frac{c}{q}\right)^n q^{-{n\choose 2}},\;\mbox{ for}\; c\neq 0,\; n\in\mathbb{N},\]
we obtain 
\[\begin{split}E_q(\pm \frac{q^{-k}}{2a(1-q)})&=\left(\mp \frac{q^{-k}}{2a};q\right)_{\infty}\\
&=\left(\pm \frac{1}{2a q}\right)^k q^{-{k \choose 2}} \left(\mp 2aq;q\right)_k (\mp 1/{2a};q)_{\infty}.
\end{split}\]
Therefore, there exists $M>0$ such that 
\[\left|\frac{E_q\left(-\frac{q^{-k}}{2a(1-q)}\right)}{\text{Sinh}_q(\frac{q^{-k}}{2a(1-q)})}\right|\leq M, \;\mbox{for all}\;k\in\mathbb{N}.\]
This is a sufficient condition for the series \eqref{ser} to be absolutely convergent.

\end{proof}

\section{$q$-analogs of Riemann's zeta function }

In this section, we show that the two $q$-analogs of the zeta function defined in  \eqref{zetadefinitions} satisfy the identities 

\be\label{IDSec4}
\zeta_q(2n)=(-1)^{n-1}2^{2n-1}\frac{\beta_{2n}(q)}{[2n]!},\;
\zeta_q^*(2n)=(-1)^{n} 2^{2n-2}\frac{\widetilde{E}_{2n-1}(q)}{[2n-1]!}.
\ee

One of the approaches to  investigate   these identities   is the relation   
\be 
\frac{J_{\nu+1}(z)}{J_{\nu}(z)}=2\sum_{n=1}^{\infty}\sigma_{2n}(\nu)z^{2n-1}, 
\ee
which is proved by Kishore~\cite{Kishore}.  Here, 
 $\sigma_{2n}(\nu)$ is the Rayleigh  function defined by 
 \be
\sigma_{2n}(\nu)=\sum_{m=1}^{\infty}(j_{\nu,m})^{-2n},\;\nu>-1, 
\ee
where $n$ is a fixed positive integer  and $(j_{\nu,m})_{m=1}^{\infty}$ is the set of positive zeros of $J_\nu(z)$.  
Clearly, $\sigma_{2n}(1/2)$ is the zeta  function at even integers.

\vskip .5 cm

A $q$-analog of the Rayleigh functions is defined in \cite{Kvit95-a} by
\be\label{RDEF} \frac{J_{\nu+1}^{(2)}(z;q^2)}{J_{\nu}^{(2)}(z;q^2)}=2\sum_{k=0}^{\infty}\sigma_{2n}(\nu;q^2)z^{2n-1}.\ee

\begin{thm}\label{thm:Ray}

 The Rayleigh function  of even orders associated with the zeros of the second Jackson $q$-Bessel function is given by 
\be\label{RLF}
\sigma_{2n}^{(2)}(\nu;q^2)= -\sum_{k=1}^{\infty}\dfrac{J_{\nu+1}^{(2)}(j_{k,\nu}(q^2);q^2)}{{J_{\nu}^{(2)}}'(j_{k,\nu}(q^2);q^2)}(j_{k,\nu}(q^2))^{-2n}\;(\nu>-1).\ee

\end{thm}
\begin{proof}

Ismail \cite{Ismail-82} proved that the $q$-Lommel polynomials $h_{\nu}(x;q)$ associated with the second Jackson $q$-Bessel function are orthogonal with respect to a discrete measure $d\alpha_{\nu}(x;q)$ with jump 
\[d\alpha_{\nu}(x;q)=\dfrac{(1-q^{\nu})A_k(\nu)}{j_{k,\nu-1}^2(q)}\quad\mbox{at}\quad x=\pm \frac{1}{j_{k,\nu-1}(q)}.\]
Moreover, 
\be\label{Ism:1} \sum_{k=0}^{\infty}A_{k}(\nu+1)\frac{z}{j_{k,\nu}^2(q)-z^2}=\frac{J_{\nu+1}^{(2)}(z;q)}{J_{\nu}^{(2)}(z;q)}.\ee

Hence if we assume that $|z|<j_{1,\nu}$ , then    expanding  $1-z^2/j_{k,\nu}^2(q)$  and equating  the coefficients of the  power of $z$  yields 
\[\sigma_{2n}^{(2)}(\nu;q)=\sum_{k=1}^{\infty}\frac{A_k(\nu+1)}{j_{k,\nu}^{2n}(q)}.\]
Also, we can prove by multiplying \eqref{Ism:1} by $z-j_{m,\nu}(q)$ and then calculating the limit as $z\to j_{m,\nu}(q)$ that
\[A_k(\nu)=\dfrac{J_{\nu}^{(2)}(j_{k,\nu-1};q)}{J_{\nu-1}'^{(2)}(j_{k,\nu-1};q)},\;
\sigma_{2n}^{(2)}(\nu;q)=-\sum_{k=1}^{\infty}\frac{1}{j_{k,\nu}^{2n}(q)}\dfrac{J_{\nu+1}^{(2)}(j_{k,\nu};q)}{J_{\nu}'^{(2)}(j_{k,\nu};q)}.\]
This proves \eqref{RLF} and completes the proof. 
\end{proof}
\vskip .5 cm 

Another proof of Theorem \ref{thm:Ray} follows from the identity 

\be \label{Sampling_1}\dfrac{J_{\nu+1}^{(2)}(x;q^2)}{J_{\nu}^{(2)}(x;q^2)}=\sum_{n=1}^{\infty}h_{n,\nu}(q^2) x^{2n-1},\ee
where 
\[h_{n,\nu}(q^2)=\sum_{k=1}^{\infty}-2\dfrac{J_{\nu+1}^{(2)}(j_{k,\nu}(q^2);q^2)}{{J_{\nu}^{(2)}}'(j_{k,\nu}(q^2);q^2)}(j_{k,\nu}(q^2))^{-2n},\]

and  $j_{k,\nu}(q^2)$ are the positive zeros of $J_{\nu}^{(2)}(x;q^2)$. This yields the current theorem at once. 
But it is worth mentioning that the  proof of  the identity \eqref{Sampling_1} is  in \cite{Ann_Ash_Man} and  based on a conjecture that $\{J_{\nu}^{(2)}(qj_{n,\nu}x;q^2)\}_{n=1}^{\infty}$ form a Riesz bases.

\begin{prop}\label{prop:4.2}
For $n\in\mathbb{N}$
\[\begin{split}\sigma_{2n}^{(2)}(1/2;q^2)=q^{-1}
2^{-2n+2}(1-q)^{-2n+1}\sum_{k=1}^{\infty}\dfrac{\text{Cos}_q\; \xi_k}{\text{Sin}'_q(\xi_k)}\frac{1}{\zeta_k^{2n}},\\
\sigma_{2n}^{(2)}(-1/2;q^2)=
-2^{-2n+2}(1-q)^{-2n+1}\sum_{k=1}^{\infty}\dfrac{\text{Sin}_q\; \eta_k}{\text{Cos}'_q(\eta_k)}\frac{1}{\eta_k^{2n}}.
\end{split}\]

\end{prop}
\begin{proof}
Using the identities,
\be\begin{split}\label{ID00}
	J_{1/2}^{(2)}(x;q^2)&= \frac{1}{\Gamma_{q^2}(1/2)} (\frac{2}{x(1-q^2)})^{1/2} \text{Sin}_q(\frac{x}{2(1-q)}),\\
  J_{3/2}^{(2)}(x;q^2)&=\frac{q^{-1}}{\Gamma_{q^2}(1/2)} (\frac{2}{x(1-q^2)})^{1/2}\left[\frac{2(1-q)}{x}\text{Sin}_q(\frac{x}{2(1-q)})-\text{Cos}_q( \frac{x}{2(1-q)})\right],
  \end{split}\ee
we can prove that 
\be \label{Eq:10}\dfrac{J_{3/2}^{(2)}(j_{k,1/2};q^2)}{{J_{1/2}^{(2)}}'(j_{k,1/2};q^2)}=-2 q^{-1}(1-q)\frac{\text{Cos}_q\; \xi_k}{\text{Sin}_q'(\xi_k)},\ee
where we used that if $j_{k,1/2}$ is the $k$th positive zero of $J_{1/2}(z;q^2)$, then $\xi_k=\frac{j_{k,1/2}(1-q)}{2}$. Consequently
 the result follows by substituting \eqref{Eq:10} into \eqref{RLF}. 
Similarly, we can prove the second identity  for $\sigma_{2n}(-1/2)$.

\end{proof}
\begin{prop}\label{prop:1} For $|x|<2\xi_1$.
\bea\label{R:2}  \frac{x}{1-q}\dfrac{\text{Cosh}_q \frac{x}{2(1-q)}}{\text{Sinh}_q\frac{x}{2(1-q)}}
&=&2-\frac{q}{1-q}\sum_{n=1}^{\infty}(-1)^n x^{2n}\sigma_{2n}^{(2)}(1/2;q),\\
\text{Tan}_q\frac{x}{2(1-q)}&=&\sum_{n=1}^{\infty}\sigma_{2n}^{(2)}(-1/2;q)(q^2) x^{2n-1}.
\eea
\end{prop}

\begin{proof}
Set $\nu=1/2$  in \eqref{RDEF} gives 
\be \label{Eq:R2} 
\dfrac{J_{3/2}^{(2)}(ix;q^2)}{J_{1/2}^{(2)}(ix;q^2)}=\sum_{n=1}^{\infty}\sigma_{2n}^{(2)}(1/2;q)(ix)^{2n-1}.
\ee
Then \eqref{R:2} follows by substituting  with \eqref{ID00} onto the left hand side of \eqref{Eq:R2}.
\end{proof}
The following result  gives a $q$-analog of the known identities 

\be \sigma_{2n}(1/2)=(-1)^{n-1}\frac{2^{2n-1}}{[2n]}\beta_{2n},\;
\sigma_{2n}(-1/2)= (-1)^n\frac{2^{2n-2}}{2n!}G_{2n},\ee
where $\beta_{n}$, and $G_n$ are the classical Bernoulli and Genocchi numbers, see~\cite{Kishore}.   

\begin{prop}\label{prop:4.4}
For $n\in\mathbb{N}$,
\bea \label{Eq:S1}\sigma_{2n}^{(2)}(1/2;q^2)&=&2(-1)^{n-1}q^{-1}(1-q)^{-2n+1} \frac{\beta_{2n}(q)}{[2n]!},\\
 \label{Eq:S2}\sigma_{2n}^{(2)}(-1/2;q^2)&=&(-1)^{n} (1-q)^{-2n+1}\frac{G_{2n}(q)}{[2n]!}.\eea
\end{prop}
\begin{proof}
Using the identity, see~\cite{Ism-Man-2019}, 

\be\label{Eq:11}
x\dfrac{\text{Cosh}_q x/2}{\text{Sinh}_q x/2}=2\sum_{n=0}^{\infty} \beta_{2n}(q)\frac{(x)^{2n}}{[2n]!},\; |x|<2\xi_1, 
\ee
and \eqref{R:2}, we conclude that 
\be\label{Eq:12}
2\sum_{n=0}^{\infty} \beta_{2n}(q)\frac{(x)^{2n}}{(q;q)_{2n}}=2-\frac{q}{1-q}\sum_{n=1}^{\infty}(-1)^n x^{2n}\sigma_{2n}^{(2)}(1/2;q),  |x|<2\xi_1.
\ee
Then equating the coefficients of $x^{2n}$ in the two sides of \eqref{Eq:12} yields the result. 

The proof of \eqref{Eq:S2} follows from the identity, see ~\cite{Ism-Man-2019}, 
\[\text{Tan}_q\left(\frac{x}{2(1-q)}\right)=\sum_{n=1}^{\infty}(-1)^n 2^{2n-1}
\widetilde{E}_{2n-1}(q)\frac{x^{2n-1}}{(q;q)_{2n-1}},\]
and 
 the fact that $G_n(q)=[n]\widetilde{E}_{n-1}(q).$

\end{proof}

\begin{cor}
For $n\in\mathbb{N}$,
\be \begin{split}\label{z1n}
\zeta_q(2n)&=(-1)^{n-1}2^{2n-1}\frac{\beta_{2n}(q)}{[2n]!},\\
\zeta_q^{*}(2n)&=(-1)^{n}2^{2n-1} \frac{\widetilde{E}_{2n-1}(q)}{[2n-1]!}.
\end{split}\ee
\end{cor}
\begin{proof}
The proof follows directly from Propositions  \ref{prop:4.2}, \ref{prop:4.4}.
\end{proof}

\section{A $q$-Dirichlet eta  function associated with a $q$-analog for the Bernoulli polynomials}
\label{Sec:Dir-Bernoulli}

In this section, we use contour integration and the generating function for the Bernoulli polynomials  to define a $q$-analog  of the Dirichlet eta function.  

\vskip .5 cm 
\begin{defn}
For $a>0$  and $s\in\mathbb{C}$,  let   $H_q(s,a)$ be the function defined by
\be\label{Hdef}
H_q(s,a)=\frac{1}{2\pi\,i}\int_{\Gamma}z^{s-1}\dfrac{e_q(az)}{\E(z)-1}\,dz,
\ee
where $\E(z)$ is the function defined in \eqref{ne} and the
 contour $\Gamma$ is a loop around the negative real axis, as shown in Figure \ref{Fig-1}. The loop is composed of three parts $C_1$, $C_2$, $C_3$.  The negatively oriented  circle $C_2$ is of radius $c < \min\{2\xi_1,\frac{1}{a(1-q)}\}$  and centered at  the origin, and $C_1$, $C_3$ are the lower and upper edges of a cut  in the $z$-plane along the negative real axis, traversed as shown in Figure \ref{Fig-1}.

\end{defn}

\begin{figure}[htbp]
    \begin{center} 
        \includegraphics[width=.6\textwidth ]{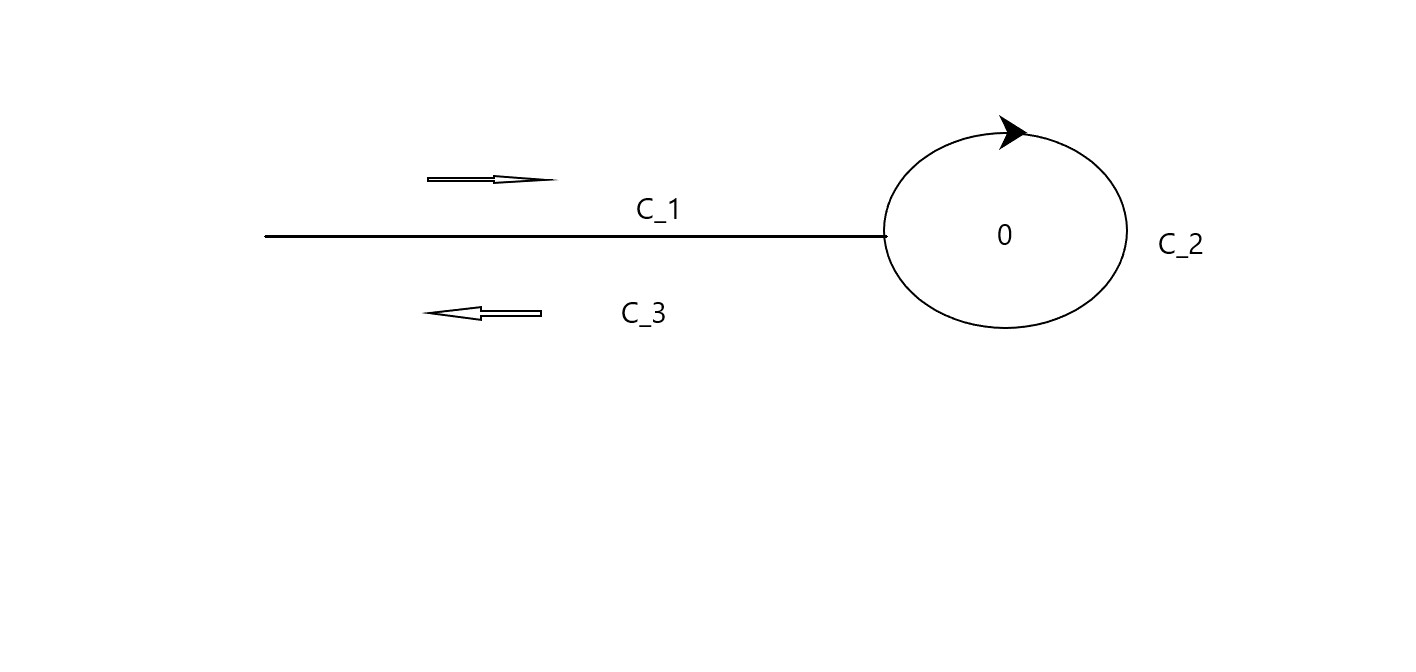} 
				\end{center}
        \caption{The contour $\Gamma$}
				\label{Fig-1}
 \end{figure}
 
\begin{thm}\label{Thm:main}
For $a>0$, the function $H_q(s,a)$ is an entire function of $s$. Moreover, 
for $\sigma:=\re s>1$,
\be \label{Eq:2-IZ}H_q(s,a)=\frac{\sin \pi  s}{\pi}\int_{0}^{\infty}r^{s-1}\dfrac{e_q(-ar)}{1-\E(-r)}\,dr,\ee
and for $n\in \mathbb{Z}$
\be\label{EQ:integer} H_q(n,a)=\frac{1}{2\pi i
}\int_{C_2}z^{n-1}\frac{e_q(az)}{\E(z)-1}\,dz.\ee

\end{thm}

\medskip

\begin{proof}

We consider an arbitrary  compact disk $|s|\leq M$ and prove that the integrals along $C_1$ and $C_3$ converge uniformly on every such disk. This approach  will prove that $H_q(s,a)$ is an entire function since the integrand is an entire function of $s$.

Along $C_1$,
we have for $r\geq 1$
\[|z^{s-1}|=|(re^{i\pi})^{s-1}|=r^{\sigma-1}e^{-\pi \tau}\leq r^{M-1} e^{\pi M},\]
on $C_3$,
\[|z^{s-1}|=|r^{s-1}e^{-\pi\,i(s-1)}|=r^{\sigma-1}e^{\pi\tau}\leq r^{M-1} e^{\pi M}.\]
Hence on either $C_1$ or $C_3$ we have for $r\geq 1$
\[
\left|z^{s-1}\frac{e_q(az)}{\E(z)-1}\right|\leq r^{M-1}e^{\pi M}\frac{e_q(-ar)}{1-\E(-r)}.
\]
We study the convergence of the limit 
\[\lim_{u\to\infty}\int_1^{u}r^{M-1}e^{\pi M}\frac{e_q(-ar)}{1-\E(-r)}\,dr=\\\lim_{u\to\infty}\int_1^{u}r^{M-1}e^{\pi M}\frac{1}{ E_q(ar )}\frac{E_q(r/2)}{ 2\,\text{Sinh}_q r/2}\,dr.\]
Since  the function $E_q(r/2)/\text{Sinh}_qr/2$ is a bounded function, see Lemma \ref{lemma1}. Moreover,  since, for any $M>0$ and $a>0$, the integral $\int_{1}^{\infty}\frac{r^{M-1}}{E_q(ar)}\,dr <\infty$, then the integrals along $C_1$ and $C_3$ converges  
uniformly on any compact disk $|s|\leq M$ and hence $H_q(s,a)$ is an entire function of $s$.
To prove \eqref{Eq:2-IZ}, note that 
\[\int_{C_1}\dfrac{z^{s-1}e_q(az)}{\E(z)-1}\, dz=-e^{i\pi(s-1)}\int_{c}^{\infty}r^{s-1}\frac{e_q(-ar)}{1-\E(-r)}\,dr,\]

\[\int_{C_3}\dfrac{z^{s-1}e_q(-az)}{\E(z)-1}\, dz=e^{-i\pi\,(s-1)}\int_{c}^{\infty}r^{s-1}\frac{e_q(-ar)}{1-\E(-r)}\,dr,\]
and 

\[\int_{C_2}\dfrac{z^{s-1}e_q(az)}{\E(z)-1}\, dz=-2\pi\,i Res(g,0)=0,\mbox{if}, \re s>1,\]
where $g(z)=\dfrac{e_q(az)}{\E(z)-1}$.
Thus  for $\re s>1$, 
\[H_q(s,a)=\frac{\sin \pi\,s}{\pi
}\int_{c}^{\infty}r^{s-1}\frac{e_q(-ar)}{1-\E(-r)}\,dr, \;c<2\xi_1.\]
 Hence, taking the limit as $c\to 0$ , we obtain 

\[H_q(s,a)=\frac{\sin \pi s}{\pi}\int_0^{\infty} r^{s-1}\frac{e_q(-ar)}{1-\E(-r)}\,dr,\quad \re s>1.\]

  If $s=n$  is an integer,  $ n\in\mathbb{N}$, then the integrals on $C_1$ and $C_3$ cancel each other and \eqref{EQ:integer} follows. This completes the proof of the theorem.
\end{proof}

\begin{prop}\label{cor:1}
For $n\in\mathbb{Z}$
\[H_q(n,a)=\left\{\begin{array}{cc}0, &n\geq 2,\\
-1, &n=1,\\
-\frac{b_{-n+1}(a;q)}{[-n+1]!},&n=0,-1,-2,\ldots.
\end{array}\right.\]
\end{prop}
\begin{proof}
If $n\in\mathbb{N}$, $n>1$,    the integrand is an analytic function and the integral in \eqref{EQ:integer} is equal to zero. 
If $n=1$, $H_q(1,a)=-Res(g(z))(0)=-1$. If $n=-m+1$, $m=1,2,\ldots$,
then  from \eqref{GF}
\[H_q(-m+1,a) =-\frac{1}{m!} (zg(z))^{(m)}(0)=-\frac{b_m(a;q)}{[m]!}.\]

\end{proof}

 \begin{figure}[ht]
   \begin{center}
        \includegraphics[width=.6\textwidth ]{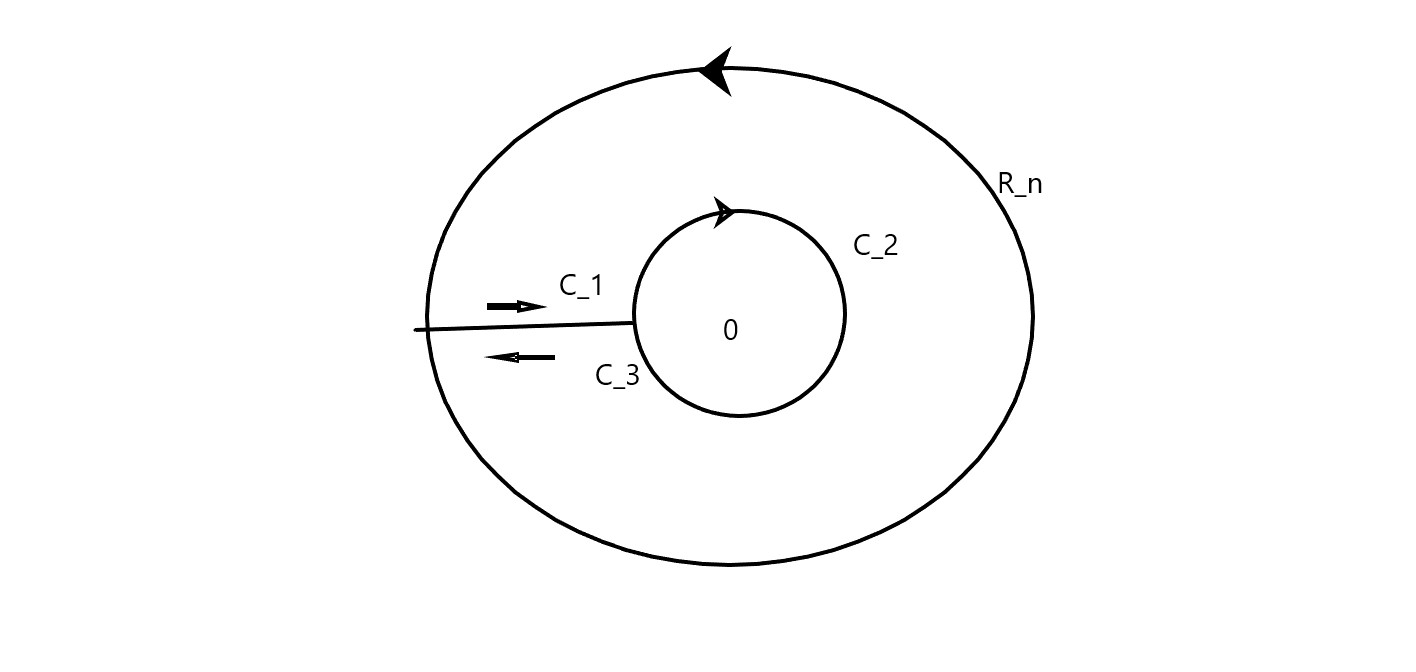} 
				\end{center}
        \caption{The contour $\Gamma_n$}
        \label{Fig2}
 \end{figure}

Now if we integrate the function $f(z)=z^{s-1}g(z)$, $g(z)=\frac{e_q(az)}{\E(z)-1}$ on the contour $\Gamma_n $, see Figure \ref{Fig2}, where the radius $R_n$ of the outer circle  $C_{R_n}$ is chosen as in Corollary \ref{Cor}. Then, 
the integral on the outer circle  of $\Gamma_n $ vanishes  as $n\to\infty$. Consequently, 
\[\lim_{n\to\infty}\int_{\Gamma_n}f(z)\,dz=\int_{\Gamma} f(z)\,dz.\] 
But 
from the Cauchy  integral formula,
\be\label{CIF}  \int_{\Gamma_n}f(z)dz=2\pi\,i\sum Res(f, z_k),\ee
where $\{z_k\}$ are the poles of $f$ that lie inside $\Gamma_n$.
This leads to the following theorem.
\begin{thm}\label{thm:bn}
For $a>0$ and $s\in\mathbb{C}$,
\[\begin{gathered}H_q(s,a)=\frac{1}{2\pi\,i}\int_{\Gamma}f(z)\,dz=2^{s-1}\sum_{k=1}^{\infty}(\pm i\xi_k)^{s-1}e_q(\pm 2ia\xi_k)\dfrac{ \text{Cos}_q\xi_k}{Sin_q'\xi_k}\\-\frac{1}{2(q;q)_{\infty}}\sum_{k=0}^{\infty}(-1)^k q^{k(k+1)/2}\left(\frac{q^{-k}}{a(1-q)}\right)^{s} \frac{E_q\left(-\frac{q^{-k}}{2a(1-q)}\right)}{\text{Sinh}_q(\frac{q^{-k}}{2a(1-q)})}.
\end{gathered}
\]
\end{thm}

\begin{proof}
The proof follows from the Cauchy's  Residue Theorem by noting that the function $f$ has simple poles at $z=\pm 2i\xi_k$ where $(\xi_k)_{k=1}^{\infty}$ is  the sequence of the positive zeros of $\text{Sin}_q\,z$ arranged in increasing order of magnitude, and at the sequence $\{\frac{q^{-k}}{a(1-q)},\; k\in\mathbb{N}_0\}$.  {The convergence of the second series on the right hand side follows from Lemma \ref{3:witha}.}
\end{proof}

Obviously, the value of the integral  $\int_{\Gamma}f(z) dz$ is independent of $c$, the radius of  the circle $C_2$,  as long as  $c< 2\xi_1$.

We define a $q$-analog  of the Dirichlet series 
\[F(s,a)=\sum_{k=1}^{\infty}\frac{e^{2\pi\,ia\,k}}{k^s},\quad \re s>0,\] by
\be\label{def-der} 
F_q(s,a):=\sum_{k=1}^{\infty}\frac{e_q(2ia\xi_k)}{\xi_k^s} \frac{\text{Cos}_q \xi_k}{\text{Sin}_q'\xi_k},\quad s\in\mathbb{C}.
\ee

One  can verify that $\lim_{q\to 1^{-}}F_q(s,a)=\frac{1}{\pi^s} F(s,a)$,
and 
\[F_q(s,q^m a)=\sum_{j=0}^{m}\qbinom{m}{j} q^{j \choose 2}(-2ia)^j F_q(s-j,a),\quad m\in\mathbb{N}.\]

\begin{prop}
For $s\in \mathbb{C}$ and $a>0$,
\[H_q(1-s,a)=(2i)^{s-1} F_q(s,a)+(-2i)^{s-1}F_q(s,-a)-R_q(s,a),\]
where 
$ F_q(s,a)$ 
 is the  $q$-analog of the Dirichlet  series defined in  \eqref{def-der} and
 \be
 R_q(s,a):=\frac{1}{2(q;q)_{\infty}}\sum_{k=0}^{\infty}(-1)^k q^{k(k+1)/2}\left(\frac{q^{-k}}{a(1-q)}\right)^{-s} \frac{E_q\left(-\frac{q^{-k}}{2a(1-q)}\right)}{\text{Sinh}_q(\frac{q^{-k}}{2a(1-q)})}. 
 \ee
\end{prop}
A simple manipulation gives 
\[R_q(s,a)+(-1)^s R_q(s,-a)=-\frac{1}{(q;q)_{\infty}}\sum_{k=0}^{\infty}(-1)^k q^{k(k+1)/2} \left(\frac{q^{-k}}{a(1-q)}\right)^{-s}.\]
\begin{prop}\label{Cor:5} For $s\in\mathbb{C}$, $a>0$

\[\begin{gathered}H_q(s,a)=2^{s}\sin \frac{\pi}{2}s\sum_{k=1}^{\infty}\frac{\xi_k^{s-1}}{(-4a^2(1-q)^2\xi_k^2;q^2)_{\infty}} \dfrac{\text{Cos}_q(2a\xi_k) \text{Cos}_q\xi_k}{Sin_q'\xi_k} \\
+2^{s}\cos \frac{\pi}{2}s\sum_{k=1}^{\infty}\frac{\xi_k^{s-1}}{(-4a^2(1-q)^2\xi_k^2;q^2)_{\infty}} \dfrac{\text{Sin}_q(2a\xi_k) \text{Cos}_q\xi_k}{Sin_q'\xi_k} 
-R_q(1-s,a).
\end{gathered}\]

\end{prop}
\begin{proof}
The proof follows from Theorem \ref{thm:bn} by calculating the real and imaginary part  of the series 
\[\sum_{k=1}^{\infty}(\pm i\xi_k)^{s-1}e_q(\pm 2ia\xi_k)\dfrac{ \text{Cos}_q\xi_k}{Sin_q'\xi_k},\] 
and noting that the imaginary part vanishes.
\end{proof}

\begin{prop} For $a>0$ and $n\in\mathbb{N}$
\be
\begin{split}
\frac{b_{n+1}(a;q)}{[n+1]!}&=
2^{-n}\sin \frac{n\pi}{2}\sum_{k=1}^{\infty}\frac{ \xi_k^{-n-1}}{(-4a^2(1-q)^2\xi_k^2;q^2)_{\infty}}\frac{\text{Cos}_q2a\xi_k\text{Cos}_q\xi_k}{\text{Sin}'_q\xi_k}\\
&-2^{-n}\cos\frac{n\pi}{2}\sum_{k=1}^{\infty}\frac{\xi_k^{-n-1}}{(-4a^2(1-q)^2\xi_k^2;q^2)_{\infty}}\frac{\text{Sin}_q2a\xi_k\text{Cos}_q\xi_k}{\text{Sin}'_q\xi_k}\\&+\frac{1 }{2(q;q)_{\infty}}\sum_{k=0}^{\infty}(-1)^k q^{k(k-1)/2}\left(\frac{q^{-k}}{a(1-q)}\right)^{-n-1} \frac{E_q\left(-\frac{q^{-k}}{2a(1-q)}\right)}{\text{Sinh}_q(\frac{q^{-k}}{2a(1-q)})}.
\end{split}
\ee
\end{prop}
\begin{proof}
The proof is a direct consequence of propositions \ref{cor:1} and \ref{Cor:5} and is omitted.

\end{proof}
 
 \begin{prop} For $a>0$ and $n\in\mathbb{N}$

\be \begin{split}
\frac{b_{2n+1}(a;q)}{[2n+1]!}&=
2^{-2n}(-1)^{n+1}\sum_{k=1}^{\infty}\frac{ \xi_k^{-2n-1}}{(-4a^2(1-q)^2\xi_k^2;q^2)_{\infty}}\frac{\text{Sin}_q2a\xi_k\text{Cos}_q\xi_k}{\text{Sin}'_q\xi_k}\\
&+\frac{1 }{2(q;q)_{\infty}}\sum_{k=0}^{\infty}(-1)^k q^{k(k-1)/2}\left(\frac{q^{-k}}{a(1-q)}\right)^{-2n-1} \frac{E_q\left(-\frac{q^{-k}}{2a(1-q)}\right)}{\text{Sinh}_q(\frac{q^{-k}}{2a(1-q)})}.
\end{split}
\ee
 
 \end{prop}
 \begin{prop}\label{prop:2.9} For $a>0$ and $n\in\mathbb{N}$
 \be 
 \begin{split}
\frac{b_{2n}(a;q)}{[2n]!}&=
(-1)^{n+1}2^{-2n+1}\sum_{k=1}^{\infty}\frac{\xi_k^{-2n}}{{(-4a^2(1-q)^2\xi_k^2;q^2)_{\infty}}}\frac{\text{Cos}_q2a\xi_k\text{Cos}_q\xi_k}{\text{Sin}'_q\xi_k}\\&+\frac{1 }{2(q;q)_{\infty}}\sum_{k=0}^{\infty}(-1)^k q^{k(k-1)/2}\left(\frac{q^{-k}}{a(1-q)}\right)^{-2n} \frac{E_q\left(-\frac{q^{-k}}{2a(1-q)}\right)}{\text{Sinh}_q(\frac{q^{-k}}{2a(1-q)})}.
\end{split}
\ee
\end{prop}

\begin{prop} For $n\in\mathbb{N}$, 
\be
\frac{b_{n+1}(1/2;q)}{[n+1]!}=2^{-n}\sin\frac{\pi\,n}{2}
\sum_{k=1}^{\infty}(\xi_k)^{-n-1}\frac{1}{\text{Sin}'_q\xi_k}.
\ee
Hence, $b_{2n+1}(1/2;q)=0$ and 
\[\frac{b_{2n}(1/2;q)}{[2n]!}=2^{-2n+1}(-1)^{n+1}\sum_{k=1}^{\infty}\frac{1}{\xi_k^{2n
}}\frac{1}{\text{Sin}'_q\xi_k}.\]

\end{prop}

In particular if $a=1/2$, we obtain 
\[\sum_{k=1}^{\infty}\dfrac{1}{Sin_q'\xi_k}=-\frac{1}{2}.\]

\begin{defn}
We define  a  $q$-analog of the Hurwitz zeta  function by
\be\label{Eq:zeta}
\zeta_q(s,a):=\Gamma_q(1-s) H_q(s,a),\quad s\in\mathbb{C}, \;a>0.
\ee
\end{defn}

\begin{thm}
The function $\zeta_q(s,a)$ is analytic for all $s$ except for a
simple pole at $s = 1$ with residue $-\frac{1-q}{\log q}$.
\end{thm}
\begin{proof}
 Since $H_q(s, a)$ is entire,  the only possible singularities of $\zeta_q(s, a)$ are the poles of $ \Gamma_q(1 -s)$, that is, the points $s = 1, 2, 3, \ldots$.  But Theorem \ref{Thm:main} shows that $\zeta_q(s,a)$ is analytic at $s = 2, 3,\ldots$, so $s=1$ is the only possible pole of $\zeta_q(s, a)$. From Corollary \ref{cor:1}, the residue of the pole at $s=1$ is equal to
 \[\lim_{s\to 1}(s-1)\Gamma_q(1-s)  H_q(s,a)=-\lim_{s\to 1} (s-1)\Gamma_q(1-s)=-\frac{1-q}{\log q}.\]
  
\end{proof}

\begin{defn}
We define a $q$-analog of the Dirichlet eta   function by
\be\label{zeta-def}
\eta_q(s)=-\sum_{k=1}^{\infty}\frac{1}{\xi_k^s}\frac{1}{\text{Sin}'_q\xi_k}
=\sum_{k=1}^{\infty}\frac{1}{\xi_k^s}\frac{(-1)^{k-1}}{|\text{Sin}'_q\xi_k|},\quad s\in\mathbb{C}.
\ee

\end{defn}

% This looks like a repetition Clearly $\eta_q(s)=-\sum_{k=1}^{\infty}\frac{1}{\xi_k^s}\frac{1}{\text{Sin}'_q\xi_k}.$
\begin{prop}

For $s\in\mathbb{C}$

  \be \label{def:5}\zeta_q(1-s,1/2)=-2^{1-s} \cos \frac{\pi}{2}s \,\Gamma_q(s)\eta_q(s).\ee

\end{prop}
\begin{proof}
The proof follows from \eqref{Eq:zeta} and  by setting $a=1/2$ in Proposition \ref{Cor:5}.
\end{proof}
\begin{prop}

\begin{eqnarray}\label{def:eta}\eta_q(2n)&=&2^{2n-1}(-1)^{n} \frac{b_{2n}(1/2;q)}{[2n]!}, \;n\in\mathbb{N}_0,\\\label{eta:negative}
\eta_q(-2n)&=&0,\quad  n\in\mathbb{N}. 
\end{eqnarray}
\end{prop}
\begin{proof}
The proof  of \eqref{def:eta} follows from Proposition \ref{prop:2.9} by the substitution $a=1/2$. To prove \eqref{eta:negative}, note that from \eqref{def:5}  and \eqref{Eq:zeta},
\[H_q(s,1/2)=-2^s \sin \frac{s\pi}{2}\,  \eta_q(1-s).\]
Consequently, from  Corollary \ref{Cor-1},

\[H_q(2n+1)=-2^{2n+1}(-1)^n \eta_q(-2n)=0, \; n\in\mathbb{N}. \] 

\end{proof}

\begin{rem}
    Ismail and Mansour in \cite{Ism-Man-2019} proved that 
    \begin{equation}\label{betan-values} \begin{gathered}\beta_1(q)=-\frac{1}{2},\;\beta_2(q)=\frac{q}{2^2}(-q;q)\frac{1}{[3]},\\\beta_4(q)=-\frac{q^4}{2^4}(-q;q)_2\frac{[2]}{[3][5]},\;\beta_6(q)=\frac{q^7}{2^6}(-q;q)_3\frac{[4]}{[3][7]}.
\end{gathered}
\end{equation}
In addition to the fact that $\beta_{2k+1}(q)=0$  for $k\in\mathbb{N}$.  They also proved that 
\be \label{bB}b_n(x;q)=\sum_{k=0}^{n}\qbinom{n}{k} \beta_{n-k}(q) x^k.\ee

 In fact, one can verify that
 \begin{eqnarray*} 
b_0(1/2;q)&=&1,\quad b_2(1/2;q)=-\frac{q^{3}}{4[3]},\\
b_4(1/2;q)&=&\frac{q^{4}}{2^4[3][5]}\left(-[2][4]+[3][5]\right)>0,\\
b_6(1/2;q)&=&\frac{q^7}{2^6[3][7]}\left[(-q;q)_3 [4]-(1+q^2+q^3)[3][7]\right]<0.
\end{eqnarray*}
Consequently, $\eta_q(2)=\frac{q^{3}}{2[3]!}>0$, 
\[\eta_q(4)=\frac{q^4}{2^6[3]([5]!)}\left([2][4]+[3][5]\right)>0,
\]

and 

\[\eta_q(6)=\frac{q^7}{2[3]\left([7]!\right)}\left[-(-q;q)_3 [4]+(1+q^2+q^3)[3][7]\right]>0. \]
% 
%In the classical case, the  alternating zeta function and the zeta function are conne
%\[\zeta^{*}(s)=(1-2^{1-s})\zeta(s).\]
%See~\cite{Apostol, Sondow}.  
 %

\medskip

   Also, The identity in \eqref{def:eta}   may look unexpected to the reader, since in the classical case we have the Bernoulli numbers on the right-hand side of \eqref{def:eta}.  However, our result is not far-removed from that  because in the classical  case there is a relation between the Bernoulli polynomials at $1/2$ and the Bernoulli numbers given by
   \[B_n(1/2)=(2^{-n+1}-1)\beta_n, \quad n\in\mathbb{N},\]
   see~\cite{Apostol}.
    
\end{rem}
\section{A $q$-Dirichlet eta functions associated with a $q$-analog of the Euler polynomials }\label{sec:Der-Euler}

In this section, we use  the generating function of the $q$-Euler polynomials
\[F_q(z,a):=\frac{2e_q(az)}{\E(z)+1}=\sum_{n=0}^{\infty}
\frac{e_n(a;q)}{[n]!}z^n,\]
and contour integration to  derive another $q$-analog of the Dirichlet eta  function,   
We consider the function defined by the contour integral
\be\label{Idef} I_q(s,a)=\frac{1}{2\pi\,i}\int_{\Gamma} z^{s-1} F_q(z,a)\, dz,\ee
where $\Gamma$ is a contour  similar to the contour in Figure \ref{Fig-1} but the radius of the circle $C_2$ is less than $\min\{2\eta_1,\frac{1}{a(1-q)}\}$, where $\eta_1$ is the smallest positive zeros of $\text{Cos}_qz$.
Similar to Theorem \ref{Thm:main}, we can prove the following theorem.
\begin{thm}
The function $I_q(s,a)$ is an  entire function and for $Re(s)>0$, we have
\be 
I_q(s,a)=\frac{2\sin \pi\,(s-1)}{\pi}\int_{0}^{\infty} r^{s-1}\frac{e_q(-ar)}{\E(-r)+1}\,dr,\quad \re s>0.
\ee
Moreover,
\be \label{Eq:SH}
\begin{split}
I_q(s,a)&=2^{s+1} \sin\frac{\pi}{2}s \sum_{k=1}^{\infty} \eta_k^{s-1}\dfrac{Cos_q(2a\eta_k) Sin_q\eta_k}{(-4a^2
(1-q)^2\eta_k^2;q^2)_{\infty} Cos_q'\eta_k}\\
&+2^{s+1} \cos\frac{\pi}{2}s \sum_{k=1}^{\infty}\eta_k^{s-1} \dfrac{Sin_q(2a\eta_k) Sin_q\eta_k}
{(-4a^2(1-q)^2\eta_k^2;q^2)_{\infty} Cos_q'\eta_k}\\
&-\frac{1}{(q;q)_{\infty}}\sum_{k=0}^{\infty}(-1)^k q^{k(k+1)/2}\left(\frac{q^{-k}}{a(1-q)}\right)^{s} \frac{E_q\left(-\frac{q^{-k}}{2a(1-q)}\right)}{\text{Cosh}_q(\frac{q^{-k}}{2a(1-q)})}.
\end{split}
\ee
\end{thm}
If $s=n$, $n$ is an integer, the integrals on the lines $C_1$ and $C_3$ cancel each other and 
\be\label{Eq:47} I_q(n,a)=\frac{1}{2\pi\,i}\int_{C_2}z^{n-1}F_q(a,z)\,dz=
\left\{\begin{array}{cc}0,& n=1,2,\ldots,\\
-1,&n=0,\\
-\frac{e_{-n}(a;q)}{[-n]!},&n=-1,-2,\ldots.
\end{array}\right.
\ee
If we set $a=1/2$ and replace $s$ by $1-s$ in \eqref{Eq:SH}, we obtain
\be\label{eta*}   I_q(1-s,1/2)=2^{2-s} \cos\frac{\pi}{2}(1-s) \sum_{k=1}^{\infty}\frac{1}{ \eta_k^{s}}\frac{1}{ Cos_q'\eta_k}.\ee

\begin{defn}
We define a $q$-analog of the alternating  zeta function or Dirichlet eta function  by
\[\eta_q^{*}(s)=-\sum_{k=1}^{\infty}\frac{1}{\eta_k^s}\frac{1}{Cos'_q\eta_k}=\sum_{k=1}^{\infty}\frac{(-1)^{k-1}}{\eta_k^s}\frac{1}{|Cos'_q\eta_k|},\;s\in\mathbb{C}.\]
\end{defn}

From \eqref{Eq:47}, we obtain
\[\frac{e_{n-1}(1/2;q)}{[n-1]!}=2^{2-n}\cos_q\frac{(n-1)\pi}{2}\,\eta_q^{*}(n).\]
In particular, 
\be \label{def:eta2}\eta_q^{*}(2n+1)=2^{2n+1}(-1)^n\frac{e_{2n}(1/2;q)}{[2n]!},\quad n\in\mathbb{N}_0.\ee

From  \eqref{eta*} 
\be \label{eta**}I_q(1-s,1/2)=2^{2-s}\cos\left( \frac{(1-s)\pi}{2} \right) \eta_q^*(s).     \ee

Therefore, substituting with $s=-2n+1$, $n\in\mathbb{N}$, in \eqref{Eq:47}, we obtain 
\[\eta_q^*(-2n+1)=0,\quad n\in\mathbb{N}.\]

\begin{thm}
For $\re s>1 $
\bea\label{r1:ze} \zeta_q(s)&=&\sum_{j=0}^{\infty}\frac{(-1)^{j+1}}{[2j]!}q^{j(2j-1)}\eta_q(s-2j),\\
\label{r2:ze}\zeta_q^*(s)&=&\frac{1}{2}\sum_{j=0}^{\infty}\frac{(-1)^{j}}{[2j+1]!}q^{j(2j+1)}\eta_q^{*}(s-2j-1).\eea
In particular, if $s=2n$, then 
\bea\label{r1:ze1} \zeta_q(2n)&=&\sum_{j=0}^{n}\frac{(-1)^{j+1}}{[2j]!}q^{j(2j-1)}\eta_q(2n-2j),\\
\label{r2:ze2}\zeta_q^*(2n)&=&\frac{1}{2}\sum_{j=0}^{n-1}\frac{(-1)^{j}}{[2j+1]!}q^{j(2j+1)}\eta_q^{*}(2n-2j-1).\eea
\end{thm}
\begin{proof}
Substituting with the series expansion of $\text{Cos}_q \xi_k$ in \eqref{S-C-def}, this gives 
\[\zeta_q(s)=-\sum_{k=1}^{\infty}\frac{1}{\xi_k^s}\frac{1}{\text{Sin}_q'(\xi_k)}\sum_{j=0}^{\infty}(-1)^j \frac{q^{j(2j-1)}}{[2j]!}\xi_k^{2j}.\]
Then changing the order of summation and noting that 
\[\eta_q(s-2j)=\sum_{k=1}^{\infty}\frac{1}{\xi_k^{s-2j}}\frac{1}{\text{Sin}_q'\xi_k},\]
yields \eqref{r1:ze}. The proof of \eqref{r2:ze} is similar and is omitted. 
The proof of  \eqref{r1:ze1} and \eqref{r2:ze2} follows from the fact $\eta_q$ vanishes at the negative even integers and
$\eta_q^*$ vanishes at the negative odd integers.  
\end{proof}
\begin{thm}
For $n\in\mathbb{N}$
\be\label{e1p} \eta_q(2n)=\frac{(-1)^{n+1}}{2[2n-1]!}+\sum_{k=0}^{n}\frac{(-1)^{k+1}}{[2k]!}\zeta_q(2n-2k).\ee
\be \label{e2p}\eta_q^*(2n+1)=\frac{(-1)^{n-1}}{2[2n]!}+\sum_{k=0}^{n-1}\frac{(-1)^{k}}{[2k+1]!}\zeta_q^{*}(2n-2k).\ee
\end{thm}
\begin{proof}
The proof of \eqref{e1p} follows by  setting $x=1/2$ and replacing $n$ by $2n$ in the identity 
\[b_n(x;q)=\sum_{k=0}^{n}\qbinom{n}{k} \beta_{n-k} x^k,\]
  and using   \eqref{def:eta} and \eqref{z1n} with taking into consideration the fact that  $\beta_{2k+1}(q)=0$ for all $k\in\mathbb{N}$.
The proof of \eqref{e2p} follows from the identity
\[e_n(x;q)=\sum_{k=0}^{n}\qbinom{n}{k}\widetilde{E}_{n-k}x^{k} ,\]
by replacing $n$ by $2n$,   $x$ by $1/2$,  using  \eqref{def:eta2} and \eqref{z1n}, and   taking into consideration the fact that $\widetilde{E}_{2k}=0$ for all  $k\in\mathbb{N}$.
\end{proof}

\section{A contour integration approach to  the $q$-analogs of the zeta function}
In  this section,  we show that the $q$-analog  of the zeta functions defined in \eqref
{IDSec4} can be derived by using contour integration as in Sections \ref{Sec:Dir-Bernoulli} and  \ref{sec:Der-Euler}. 
We start with  the function $H_q(s,0)$ defined in \eqref{Hdef}.  In this  case, the contour $\Gamma$ is the same as in Fig~\ref{Fig-1}, but the radius $c$ of the circle $C_2$ is less than $2\xi_1$.  We can prove the following result:
\begin{thm}
The function $H_q(s,0)$ is analytic for $\re s<0$ and has the representation 
\be\label{Eq:0}
H_q(s,0)=\frac{\sin\pi s}{\pi}\int_{c}^{\infty} r^{s-1}\frac{dr}{1-\varepsilon_q(-r)}+\frac{1}{2\pi\,i}\int_{C_2} z^{s-1}\frac{1}
{\varepsilon_q(z)-1}\,dz,
\ee
where $\E(z)$ is the function defined in \eqref{ne}.
\end{thm}

\begin{proof}
We can show that on $C_1$ and $C_3$, the function  $\frac{1}{1-\varepsilon_q(-r)}$ is bounded for $r\geq c$, and therefore the integral $\int_{c}^{\infty}r^{\re s-1} \frac{1}{1-\varepsilon_q(-r)}\, dr$ converges for $\re s<0$ and \eqref{Eq:0} follows.
\end{proof}
Consequently, if $s$ is an integer, then 
\be 
H_q(n,0)=\frac{1}{2\pi\,i}\int_{C_2} z^{n-1}\frac{1}{1-\varepsilon_q(z)}\,dz=\left\{\begin{array}{cc} 0,&\, n=2,3,\ldots,\\1,&n=1,\\
-\frac{\beta_{-n+1}(q)}{[n+1]!},& n=-1,-2,\ldots. 
 \end{array}\right.
\ee
Now we consider the integral on $\Gamma_n$ as in Figure \ref{Fig2} but $R_n$ is chosen to satisfy $\xi_n<R_n<\xi_{n+1}$. So
we can prove that  for $\re s<0$,
\be \begin{split}\lim_{n\to\infty}\int_{\Gamma_n}z^{s-1}\frac{1}{\varepsilon_q(z)-1}\,dz&=\int_{\Gamma}z^{s-1}\frac{1}{\varepsilon_q(z)-1}\,dz\\
&= 2\pi \,i\lim_{n\to\infty}\sum_{k=0}^{n}\text{Residue} \left(z^{s-1}\frac{1}{\varepsilon_q(z)-1}\right)(\pm 2i\xi_k)\\
&=2\pi \,i\lim_{n\to\infty}2^{s}\sin\frac{\pi}{2}s\sum_{k=1}^{n}\xi_{k}^{s-1}\frac{\text{Cos}_q\xi_k}{\text{Sin}'_q \xi_k}\\
&=2\pi \,i2^{s}\sin\frac{\pi}{2}s\sum_{k=1}^{\infty}\xi_{k}^{s-1}\frac{\text{Cos}_q\xi_k}{\text{Sin}'_q \xi_k}
\end{split}\ee
Therefore, for $\re s>1$.

\be H_q(1-s,0)=2^{1-s}\sin\frac{\pi}{2}(1-s)\zeta_q(s).\ee

This leads to the result 
\[\zeta_q(2n)=(-1)^{n-1} 2^{2n-1} H_q(1-2n,0)=(-1)^n 2^{2n-1} \frac{\beta_{2n}}{[2n]!},\]

which coincides with \eqref{z1n}.
Similarly,  we calculate the contour integral $I_q(s,a)$ defined in \eqref{Idef} when $a=0$. I.e.

\[I_q(s,0)=\frac{1}{2\pi\,i}\int_{\Gamma }z^{s-1}\frac{2}{\varepsilon_q(z)+1}\, dz,\]
where $\Gamma$ is the contour defined in Figure \ref{Fig-1} but the radius $c$ of the circle $C_2$ is less that $2\eta_1$. 

If $\re s<0$, then $I_q(s,0)$ is analytic and can be represented as 
\[I_q(s,0)=2\frac{\sin \pi\,s}{\pi}\int_{c}^{\infty} \frac{r^{s-1}}{\varepsilon_q(-r)+1 }\,dr+\frac{1}{2\pi\,i}\int_{C_2}z^{s-1}\frac{2}{\varepsilon_q(z)+1}\,dz.\]
If $s$ is an integer, then 
\be
 \label{Eqf}I_q(n,0)=\left\{\begin{array}{cc}zero,& n\in\mathbb{N},\\
  2,& n=0\\ -\frac{\widetilde{E}_{-n}}{[-n]!},& n=-1,-2,-3,\ldots\end{array}\right.
  \ee
  
If we consider the contour $\Gamma_n$ in Figure \ref{Fig2}, where the radius $R_n$ of the outer circle satisfies 
$\eta_n< R_n<\eta_{n+1}$.
We can prove that  for $\re s<0$, 
\[\lim_{n\to\infty}\int_{\Gamma_n}z^{s-1}\frac{2}{\varepsilon_q(z)+1}=I_q(s,0).\]
Therefore, using the Cauchy Residue Theorem, we can prove that 
\[I_q(s,0)=2^{s+1}\sin \frac{\pi}{2}s \sum_{k=1}^{\infty}\eta_k^{s-1}\frac{Sin_q\eta_k}{Cos'_q\eta_k},\;Re s<0.\]
Therefore,
\[I_q(1-s,0)=-2^{2-s}\sin\frac{\pi}{2}(1-s)\;\zeta_q^*(s),\;\re s>1.\]
Consequently, from \eqref{Eqf}
\[ \zeta_q^{*}(2n)=(-1)^n 2^{2n-2}\frac{\widetilde{E}_{2n-1}}{[2n-1]!},\]
which coincides with \eqref{z1n}.

% ----------------------------------------------------------------

%\section{About references}

\subsection*{Acknowledgments}Zeinab Mansour would like to thank Abdus Salam  International Center for Theoretical  Physics, Trieste, Italy, for support and hospitality  during  this work. We  would like to thank Professor Mourad Ismail for many helpful discussions and for drawing our  attention to the reference \cite{Ismail-82}.  %

\end{document}